\documentclass[12pt,reqno,centering]{amsart}
\numberwithin{equation}{section}
\usepackage{amssymb}
\usepackage{fullpage}
\usepackage{stackengine}
\usepackage{scalerel}
\usepackage{tikz}
\usetikzlibrary{patterns}
\definecolor{linkred}{rgb}{0.7,0.2,0.2}
\definecolor{linkblue}{rgb}{0,0.2,0.6}
\definecolor{linkgreen}{rgb}{0,0.6,0.2}
\usepackage[colorlinks,plainpages,backref,
    linkcolor=linkblue,
    citecolor=linkgreen,
    urlcolor=linkred]{hyperref}

\usepackage[all]{xy}
\usepackage{xcolor}
\usepackage{tikz}

\allowdisplaybreaks

\RequirePackage[sc]{mathpazo}
\RequirePackage[T1]{fontenc}
\usepackage{eucal}

\def\wt{\operatorname{wt}}
\def\PP{\operatorname{PP}}
\newtheorem{Thm}{Theorem}

\newtheorem{conj}[Thm]{Conjecture}

\newtheorem{theorem}{Theorem}[section]
\newtheorem{lemma}[theorem]{Lemma}
\newtheorem{prop}[theorem]{Proposition}
\newtheorem{coro}[theorem]{Corollary}
\newtheorem{example}[theorem]{Example}
\newtheorem{definition}[theorem]{Definition}
\newcommand{\BPD}[2][1.2pc]{%
\setlength{\unitlength}{#1}
\definecolor{lightcyan}{rgb}{0.8,1,1}%
\def\FF{%
    \qbezier(0.5,0)(0.5,0.2)(0.5,0.2)
    \qbezier(1,0.5)(0.8,0.5)(0.8,0.5)
    \qbezier(0.8,0.5)(0.5,0.5)(0.5,0.2)    
}
\def\JJ{%
    \qbezier(0.5,1)(0.5,0.8)(0.5,0.8)
    \qbezier(0,0.5)(0.2,0.5)(0.2,0.5)
    \qbezier(0.5,0.8)(0.5,0.5)(0.2,0.5)
    }
\def\II{%
    \qbezier(0.5,0)(0.5,0.5)(0.5,1)}%
\def\HH{%
    \qbezier(0,0.5)(0.5,0.5)(1,0.5)}%
\def\XX{\NN\HH}
\def\NN{%
    \qbezier(0.5,0)(0.5,0.3)(0.5,0.3)
    \qbezier(0.5,1)(0.5,0.7)(0.5,0.7)}%
\def\LL{%
    \qbezier(0.5,1)(0.75,0.75)(1,0.5)}%
\def\ZZ{%
    \qbezier(0.5,0)(0.25,0.25)(0,0.5)}%
\def\BPDfr##1{%
\begin{picture}(1,1)%
    \linethickness{0.08\unitlength}
    ##1
    \thinlines
    \color{lightgray}%
    \put(0,0){\line(0,1){1}}%
    \put(1,0){\line(0,1){1}}%
    \put(0,0){\line(1,0){1}}%
    \put(0,1){\line(1,0){1}}%
\end{picture}}
\def\BPDfrc##1{\BPDfr{\put(0,0){\color{lightcyan}\rule{\unitlength}{\unitlength}}##1}}
\def\x{\BPDfrc{\XX}}
\def\f{\BPDfrc{\FF}}
\def\o{\BPDfrc{}}
\def\j{\BPDfrc{\JJ}}
\def\h{\BPDfrc{\HH}}
\def\i{\BPDfrc{\II}}
\def\O{\BPDfr{}}
\def\X{\BPDfr{\XX}}
\def\BX{\BPDfr{\II\HH}}
\def\F{\BPDfr{\FF}}
\def\J{\BPDfr{\JJ}}
\def\H{\BPDfr{\HH}}
\def\I{\BPDfr{\II}}
\def\L{\BPDfr{\LL}}
\def\Z{\BPDfr{\ZZ}}
\def\b{\BPDfr{\LL\ZZ}}
\def\B{\BPDfr{\JJ\FF}}
\def\M##1{\begin{picture}(1,1)%
    \put(0,0.2){\makebox[\unitlength]{$##1$}}
\end{picture}}
\begin{array}{@{\,}c@{\,}}
\def\dots{\scriptstyle\cdots}
{\def\arraystretch{0}
\setlength{\arraycolsep}{0pc}
\color{teal}
\begin{array}{@{}l@{}}%
#2\end{array}}
\end{array}}

\newcommand{\clan}[2][1.2pc]{%
\setlength{\unitlength}{\dimexpr#1/2}%
\def\drawclan##1{%
\expandafter%
    \drawclanbegin.##1%
    \drawclanend\drawclanendend}
\def\drawclanbegin.##1##2\drawclanendend{%
    \ifx\drawclanend##1%
        {\relax}%
    \else%
    \clandot{##1}%
    \expandafter%
        \drawclanbegin.##2\drawclanendend%
\fi}
\def\clandot##1{%
    \if+##1
    \begin{picture}(2,1)
        \linethickness{0.15\unitlength}
        \color{red}
        \qbezier(0.4,0.4)(1,0.4)(1.6,0.4)
        \qbezier(1,1.0)(1,0.4)(1,-.2)
    \end{picture}
    \else\if-##1
    \begin{picture}(2,1)
        \linethickness{0.15\unitlength}
        \color{blue}
        \qbezier(0.5,0.4)(1,0.4)(1.5,0.4)
    \end{picture}
    \else\if.##1
    \begin{picture}(2,1)
        \color{teal}
        \put(1,0.4){\circle*{0.7}}
    \end{picture}
    \else
        \def\inpA{##1}\def\inpB{\dots}%
    \ifx\inpA\inpB%
    \begin{picture}(2,1)
        \color{teal}
        \put(0.5,0.4){\circle*{0.3}}
        \put(1,0.4){\circle*{0.3}}
        \put(1.5,0.4){\circle*{0.3}}
    \end{picture}
    \else
    \def\inpA{##1}\def\inpB{}%
    \ifx\inpA\inpB%
        \relax%
    \else
        \def\inpA{##1}\def\inpB{\pm}%
    \ifx\inpA\inpB%
    \begin{picture}(2,1)
        \linethickness{0.15\unitlength}
          \color{red}
          \qbezier(0.4,0.4)(1,0.4)(1.6,0.4)
          \qbezier(1,1.0)(1,0.4)(1,-.2)
          \color{blue}
          \qbezier(0.4,-.2)(1,-.2)(1.6,-.2)
    \end{picture}
    \else
        \def\inpA{##1}\def\inpB{\mp}%
    \ifx\inpA\inpB%
    \begin{picture}(2,1)
        \linethickness{0.15\unitlength}
        \color{red}
        \qbezier(0.4,0.4)(1,0.4)(1.6,0.4)
        \qbezier(1,1.0)(1,0.4)(1,-.2)
        \color{blue}
        \qbezier(0.4,1.0)(1,1.0)(1.6,1.)
    \end{picture}
    \else\if,##1
    \begin{picture}(2,1)
        \linethickness{0.15\unitlength}
        \color{teal!50!white}
        \put(1,0.4){\circle*{0.7}}
    \end{picture}
    \else
    \begin{picture}(2,1)
        \color{white}\linethickness{4pt}
        \qbezier(1,0.4)({\numexpr(##1)+1},{\numexpr(##1)+1})({\numexpr(##1)*2+1},0.4)%
        \color{black}\linethickness{0.8pt}
        \qbezier(1,0.4)({\numexpr(##1)+1},{\numexpr(##1)+1})({\numexpr(##1)*2+1},0.4)%
        \color{teal}
        \put(1,0.4){\circle*{0.7}}
    \end{picture}
    \rule{0pc}{\dimexpr\unitlength*(##1)/2+\unitlength}%
    \fi\fi\fi\fi\fi\fi\fi\fi%
}%
{\drawclan{#2}}%
}%

\def\myclan#1{{\hspace{-1pc}
    \begin{array}{c}\\[-1pc]
    \clan{#1}\vphantom{\dfrac12}\\\end{array}
    \hspace{-1pc}}}
\def\fmyclan#1{{\hspace{-.5pc}\begin{array}{|c|}\hline\\[-1pc]
    \hspace{-.5pc}\rule{0pc}{1.8pc}
    \clan{#1}\hspace{-.5pc}
    \\\hline\end{array}\hspace{-.5pc}}}

\DeclareFontEncoding{LS1}{}{}
\DeclareFontSubstitution{LS1}{stix}{m}{n}
\DeclareSymbolFont{stixletters}{LS1}{stix}{m}{it}

\DeclareMathSymbol{\stixcevhead}{\mathord}{stixletters}{"91}
\DeclareMathSymbol{\stixcevext}{\mathord}{stixletters}{"99}

\makeatletter

\newcommand{\stixcevfill}{%
  $\m@th
    \stixcevhead
    \hskip-.23em plus .23em
    \cleaders\hbox{$\stixcevext$}\hfill
  $%
}

\newcommand{\stixcevaux}[2]{%
  \vbox{%
    \m@th
    \ialign{##\crcr
      \stixcevfill\crcr
      \noalign{\kern-1ex\nointerlineskip}
      $\hfil#1#2\hfil$\crcr
    }%
  }%
}

\newcommand{\cev}[1]{%
  \mathpalette\stixcevaux{#1}%
}

\makeatother

\DeclareMathAccent{\vec}{\mathord}{stixletters}{"92}

\newcommand{\supp}{\text{supp}}
\newcommand{\SVan}{\text{SchubertVanishing}}
\newcommand{\eSVan}{\text{equi-SchubertVanishing}}
\newcommand{\tSVan}{\text{tri-SchubertVanishing}}

\begin{document}

\def\supp{\operatorname{\mathcal{S}upp}}
\def\emph#1{{\sf\color{green!30!black} \textbf{#1}}}

\title{ Non-vanishing of 
Single, Double, and Triple\\
 Schubert Structure Constants}

\author{Yiming Chen}
\address[Yiming Chen, Neil J.Y. Fan, Ming Yao]{Department of Mathematics, 
Sichuan University, Chengdu, Sichuan 610065, P.R. China}
\email{ym\_chen@stu.scu.edu.cn, fan@scu.edu.cn, yaom@stu.scu.edu.cn}

\author{Neil J.Y. Fan}
\author{Rui Xiong}
\address[Rui Xiong]{Department of Mathematics and Statistics, University of Ottawa, 150 Louis-Pasteur, Ottawa, ON, K1N 6N5, Canada}
\email{rxion043@uottawa.ca}
\author{Ming Yao}

\maketitle

\vspace{-.5cm}
\begin{abstract}
The Schubert vanishing problem  asks whether the single Schubert coefficients $c_{u,v}^w$ are zero.  In this paper, we consider the non-vanishing problems of double  Schubert coefficients $c_{u,v}^w(t)$ and  triple Schubert coefficients $c_{u,v}^w(t;y)$. We show that the non-vanishing of  $c_{u,v}^w(t;y)$ is completely determined by the non-vanishing of single Schubert coefficients. As a byproduct, we obtain the saturation property of the triple Littlewood--Richardson coefficients $c_{\lambda,\mu}^\nu(t;y)$. Moreover, we pose a conjecture asserting that the non-vanishing of $c_{u,v}^w(t)$ is also   determined by the non-vanishing of single or  triple Schubert coefficients. We prove a one-side inclusion of the  conjecture. For the reverse inclusion, we show that the  conjecture holds for the following three cases: the Pieri  case, the separated descents case, and the inverse Grassmannian case. 
\end{abstract}

\setcounter{tocdepth}{1}
\tableofcontents

\section{Introduction}

Schubert polynomials $\mathfrak{S}_w(x) \in \mathbb{Z}[x_1, x_2, \dots]$, indexed by permutations $w \in S_\infty$, were introduced by Lascoux and Sch\"utzenberger \cite{LS82} to represent the cohomology classes of Schubert varieties in the complete flag variety.
The  structure constants arising from their multiplication, known as the Schubert coefficients $c_{u,v}^w \in \mathbb{Z}_{\ge0}$, are defined via the expansion:
$$
\mathfrak{S}_u(x) \cdot \mathfrak{S}_v(x) = \sum_{w \in S_\infty} c_{u,v}^w\cdot \mathfrak{S}_w(x).
$$
A central open problem in algebraic combinatorics and enumerative geometry is to provide a positive rule (a combinatorial interpretation) for these  coefficients.

Double Schubert polynomials $\mathfrak{S}_w(x; t)$ are  representatives of Schubert classes in the  equivariant cohomology ring of the flag variety.  
The structure constants of equivariant Schubert calculus are the double Schubert coefficients $c_{u,v}^w(t)$, which are polynomials in $t = (t_1, t_2, \dots)$ with integer  coefficients:
$$
\mathfrak{S}_u(x; t) \cdot \mathfrak{S}_v(x; t) = \sum_{w \in S_\infty} c_{u,v}^w(t) \cdot\mathfrak{S}_w(x; t).
$$
The double Schubert coefficients $c_{u,v}^w(t)$ are known to be Graham-positive \cite{Graham01}, meaning they are polynomials in $(t_{i+1} - t_i)$ with non-negative coefficients.

While an explicit combinatorial formula remains elusive, deciding the qualitative non-vanishing of these coefficients is a  significant   task with direct  theoretical and algorithmic implications. As Knutson \cite[\S 1.4]{Knu22} underscored:  {\it for applications (including real-world engineering applications) it is more
important to know that some structure constant is positive, than it is to know its actual value.} From the perspective of Geometric Complexity Theory \cite{Mul09, MNS12,Panova},  the non-vanishing of structure constants lies at the core of algorithmic algebraic geometry. In particular, Mulmuley \cite[\S 3.7]{Mul09} identified Schubert coefficients as: {\it  one of the fundamental structural constants whose vanishing needs to be understood.}

To systematically study the non-vanishing behavior of these structure constants, 
we define the \emph{single support} and \emph{double support}:
$$\begin{aligned}
\supp^1(u,v) & =\big\{w\in S_\infty \mid c_{u,v}^w\neq 0\big\},\\
\supp^2(u,v) & =\big\{w\in S_\infty \mid c_{u,v}^w(t)\neq 0\big\}.\\
\end{aligned}$$
We conjecture  that the double support is actually determined by the single support.

\begin{conj}\label{conj:123conj}
For any $u,v\in S_\infty$, 
\begin{align}\label{eq:1->2}
\supp^2(u,v)=\bigcup_{\scriptstyle\text{reduced }u=u_1\cdot u_2
\atop \scriptstyle\text{reduced }v=v_1\cdot v_2}
\supp^1(u,v_2)\cap \supp^1(u_2,v).
\end{align}
\end{conj}

Moreover, we expect that Conjecture \ref{conj:123conj} also holds in all other types. 
The main purpose of this paper is to prove the inclusion ``$\subseteq$'' in \eqref{eq:1->2} holds in general, and prove the reverse inclusion ``$\supseteq$'' in three special cases. 

\subsection{The ``$\subseteq$'' part via triple Schubert calculus}
The proof of the inclusion ``$\subseteq$'' in \eqref{eq:1->2} is based on the triple Schubert calculus,  i.e., the product of two Schubert polynomials with two different secondary variables
\begin{equation}\label{eq:triplecoeff}
\mathfrak{S}_u(x;y)\cdot \mathfrak{S}_{v}(x;t)
=\sum_{w\in S_\infty} c_{u,v}^w(t;y)
\cdot\mathfrak{S}_{w}(x;t). 
\end{equation}
The triple Schubert calculus on Grassmannians was studied by Molev and Sagan \cite{MS},   Knutson and Tao \cite{KT03} considered the situation on flag varieties; see also 
\cite{Wheeler,Samuel,FGXpuzzle,Gaox,S-orbit}.
We also consider the \emph{triple support}:
$$
\supp^3(u,v) =\big\{w\in S_\infty \mid c_{u,v}^w(t;y)\neq 0\big\}.
$$
Note that different from the single and the double cases, the roles of $u$ and $v$ in \eqref{eq:triplecoeff} are not symmetric, i.e., $\supp^3(u,v)\neq \supp^3(v,u)$ in general. 
The following Conjecture \ref{conjb} captures this asymmetry. 

\begin{conj}\label{conjb} 
For any $u,v\in S_\infty$, 
\begin{align}\label{eq:3->2}
\supp^2(u,v)=\supp^3(u,v)\cap \supp^3(v,u).
\end{align}
\end{conj}

We remark that Conjecture \ref{conjb} is inspired by our previous work \cite[\S 7.2]{S-orbit}. 
Precisely, when $u,v$ are  inverse Grassmannian permutations, both the double and triple Schubert coefficients have explicit combinatorial formulas, in this case it is direct to check \eqref{eq:3->2}; see Section \ref{section5}.

\begin{theorem}\label{thm:123conj'}
 Conjecture \ref{conj:123conj} is equivalent to Conjecture \ref{conjb}.
In particular, the inclusion ``$\subseteq$'' part of  \eqref{eq:1->2} is true. 
\end{theorem}

The following diagram summarizes the relations among the single, double and triple supports, where the two dashed arrows are conjectural. 
$$
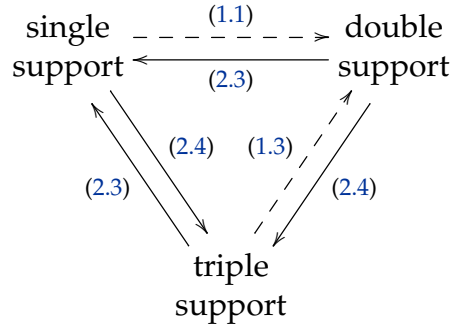
\begin{matrix}
\xymatrix@=1pc{
{\begin{matrix}
\text{single}\\
\text{support}
\end{matrix}}
\ar@{-->}@<+0.75ex>[rr]^{\text{\eqref{eq:1->2}}}
\ar@{->}@<+0.75ex>[dr]^{\text{\eqref{eq:i->3}}}
&\rule[-8pc]{0pc}{0pc}&
{\begin{matrix}
\text{double}\\
\text{support}
\end{matrix}}
\ar@{->}@<+0.75ex>[ll]^{\text{\eqref{eq:i->1}}}
\ar@{->}@<+0.75ex>[dl]^{\text{\eqref{eq:i->3}}}
\\
& 
{\begin{matrix}
\text{triple}\\
\text{support}
\end{matrix}}
\ar@{->}@<+0.75ex>[ul]^{\text{\eqref{eq:i->1}}}
\ar@{-->}@<+0.75ex>[ur]^{\text{\eqref{eq:3->2}}}
}
\end{matrix}
$$

\begin{example}\label{eg:runningeg}
Let $u=312=s_2s_1, v=231=s_1s_2$. Then the following diagram illustrates the relations among $\supp^2(u,v),\supp^3(u,v),\supp^3(v,u)$. 
$$
\def\centerarc[#1](#2)(#3:#4:#5);%
{\draw[#1]([shift=(#3:#5)]#2) arc (#3:#4:#5);}
\begin{tikzpicture}[scale=1.5]
    \begin{scope}
        \clip (0,1) ellipse (1.75 and 1.25);
        \fill[lightgray!50!white, ultra thick,rotate = 45] (0,1.4) ellipse (2 and 1);
    \end{scope}
\node at (0,0)  {$312\color{gray}45$};
\node at (1,1)  {$4123\color{gray}5$};
\node at (-1,1) {$321\color{gray}45$};
\node at (-2,0) {$231\color{gray}45$};
\node at (0,2)  {$4213\color{gray}5$};
\draw[ultra thick,gray,dashed] (-3,2.5)--+(5,0); 
\draw[ultra thick,gray,dashed] (-3,1.5)--+(5,0); 
\draw[ultra thick,gray,dashed] (-3,0.5)--+(5,0); 
\draw[ultra thick,gray,dashed] (-3,-0.5)--+(5,0); 
\draw[blue, ultra thick,rotate = 45] (0,1.4) ellipse (2 and 1);
\draw[red, ultra thick] (0,1) ellipse (1.75 and 1.25);
\fill[lightgray!50!white, ultra thick] (-.5,3) ellipse (1 and 0.3);
\node at (-.5,3) {$\supp^2(u,v)$};
\node[blue] at (-4,3) {$\supp^3(u,v)$};
\node[blue] at (-4,2.5) {$\parallel$};
\node[blue] at (-4,2) {$\supp^1(u,v)$};
\node[blue] at (-4,1.5) {$\cup$};
\node[blue] at (-4,1) {$\supp^1(s_1,v)$};
\node[blue] at (-4,0.5) {$\cup$};
\node[blue] at (-4,0) {$\supp^1(e,v)$};
\node[red] at (3,3) {$\supp^3(v,u)$};
\node[red] at (3,2.5) {$\parallel$};
\node[red] at (3,2) {$\supp^1(v,u)$};
\node[red] at (3,1.5) {$\cup$};
\node[red] at (3,1) {$\supp^1(s_2,u)$};
\node[red] at (3,0.5) {$\cup$};
\node[red] at (3,0) {$\supp^1(e,u)$};
\end{tikzpicture}$$
\end{example}

\subsection{The ``$\supseteq$'' part via 
combinatorial formulas}

We show the reverse inclusion ``$\supseteq$''  of \eqref{eq:3->2} holds in  three cases, whose combinatorial interpretations of the Schubert coefficients are known. 

\begin{theorem}\label{thm:reversepart}
 Conjecture \ref{conj:123conj} ($\Leftrightarrow$ Conjecture \ref{conjb}) is true in the following cases.
\begin{enumerate}
   
\item The \emph{Pieri case}, i.e., one of $u,v$ is of the form $s_{a}s_{a+1}\cdots s_b$ or 
$s_{b}s_{b-1}\cdots s_a$ for  
$a\leq b$;

\item 
\label{case:sep}
The \emph{separated descents case}, i.e.,  $u,v$ are of separated descents, including the case when $u,v$ are both Grassmannian permutations.

\item The \emph{inverse Grassmannian case}, i.e., $u,v$ are both inverse Grassmannian permutations. 
  
\end{enumerate}
\end{theorem}

Our proof utilizes the explicit combinatorial formulas in each case. 
The following table summaries the references of the cases above. 
$$
\begin{array}{c|c}\hline
\vphantom{\dfrac12}
\text{Cases} & 
\makebox[9pc]{single} 
\makebox[9pc]{double} 
\makebox[9pc]{triple}\\\hline
\vphantom{\dfrac12}
\text{\tiny Pieri formula}& 
\makebox[9pc]{\tiny Sottile \cite{Sottile}}
\makebox[9pc]{\tiny Robinson \cite{Robinson}; see also \cite{LSY}
} 
\makebox[9pc]{\tiny Samuel \cite{Samuel}}
\\\hline
\vphantom{\dfrac12}
\text{\tiny separated descents}& 
\makebox[18pc]{\tiny ------ Knutson--Zinn Justin
\cite{Pzz3}; see also \cite{Huang21} ------}
\makebox[9pc]{\tiny Fan--Guo--Xiong \cite{FGXpuzzle}, Samuel \cite{Samuel}}
\\\hline
\vphantom{\dfrac12}
\text{\tiny inverse Grassmannain}& 
\makebox[9pc][l]{\tiny Pechenick--Weigandt \cite{Pandt}; see also \cite{Wyser}}
\makebox[18pc]{\tiny ------ Chen--Fan--Xiong--Yao \cite{S-orbit} ------}\\\hline
\end{array}$$

When restricting to Grassmannian permutations, the Schubert coefficients are known as the Littlewood--Richardson (LR) coefficients, which are indexed by partitions. 
The proof of \eqref{case:sep} in Theorem \ref{thm:reversepart} can be applied to obtain 
a description of $\supp^3(\lambda,\mu)$, i.e. the non-vanishing of $c_{\lambda,\mu}^{\nu}(t;y)$.

Knutson and Tao \cite{KT} showed that the single LR-coefficient $c_{\lambda,\mu}^{\nu}$ has saturation property, that is, $c_{\lambda,\mu}^{\nu}\neq 0$ if and only if $c_{N\lambda,N\mu}^{N\nu}\neq 0$ for any positive integer $N$. In particular, the saturation property of $c_{\lambda,\mu}^{\nu}$ is equivalent to Horn's inequalities and the existence of Hermitian matrices (Theorem \ref{thm:saturation3}). 
 We show that the triple LR-coefficients $c_{\lambda,\mu}^{\nu}(t;y)$ also have {saturation property} (Theorem \ref{triplesaturation}).
This can be viewed as a generalization of the saturation property of double LR-coefficients  $c_{\lambda,\mu}^{\nu}(t)$ due to Anderson,  Richmond and Yong \cite{ARY}; see also \cite{RYY}.

\subsection{Applications to complexity theory}

The \emph{Schubert vanishing problem} is defined as 
 $$\SVan := \left\{ c_{u,v}^w \stackrel{?}{=} 0 \right\},$$
which is a major problem in both algebraic combinatorics and computational complexity  theory.  A recent series of works by  Pak and Robichaux   addressed the long-standing complexity status of this  problem \cite{PR24a,PR24b, PR25a, PR25b}. 
Specifically, they \cite{PR25a} showed that the Schubert vanishing problem can be decided in co-randomized polynomial time, that is,
$\SVan(Y) \in \mathsf{coRP}$
holds unconditionally for all classical Lie types $Y \in \{A, B, C, D\}$.

We also consider the \emph{equivariant Schubert vanishing problem}:  
$$\eSVan := \left\{ c_{u,v}^w(t) \stackrel{?}{=} 0 \right\},$$
and the \emph{triple Schubert vanishing problem}:
$$
\tSVan := \left\{ c_{u,v}^w(t;y) \stackrel{?}{=} 0 \right\}.
$$
For the case of Grassmannians,
Adve, Robichaux, and Yong \cite{ARY19} proved that the vanishing of double LR-coefficients can be decided in strongly polynomial time.

Since classical Schubert vanishing is unconditionally in $\mathsf{coRP}$ \cite{PR25a},  each required positive non-vanishing test admits a polynomial-length certificate, namely an accepting random string. Assuming Conjecture \ref{conj:123conj}, the polynomial-size witness tuple $(u_1, u_2, v_1, v_2)$  together with accepting random strings for the required classical non-vanishing tests, form an $\mathsf{coNP}$ certificate for the   vanishing of double and triple  Schubert coefficients. Hence, the equivariant and triple Schubert vanishing problems lie in $\mathsf{coNP}$.  
\begin{coro}
If Conjecture \ref{conj:123conj} is true, then the  equivariant Schubert vanishing problem lies in $\mathsf{coNP}$. Moreover,  the triple Schubert vanishing problem lies in $\mathsf{coNP}$ unconditionally.
\end{coro}

The organization of this paper is as follows. In Section \ref{section2}, we recall some basic notation and properties, and show that Theorem \ref{thm:123conj'}  holds. In Sections \ref{section3}, \ref{section4}, and \ref{section5}, we prove Theorem \ref{thm:reversepart}. In Section \ref{section6}, we derive  the saturation property and the non-vanishing of the triple LR-coefficients via Horn's inequalities.

\subsection*{Acknowledgment}
We are grateful to Colleen Robichaux and Alexander Yong for helpful discussions. This work is supported by the National Key Research and Development Program of China (No. 2025YFA1017702) and the National Natural Science Foundation of China (No. 12471314).

\section{Basic Properties}\label{section2}
Let $w_0=n\cdots 21\in S_n$ be the longest permutation. The double Schubert polynomials $$\mathfrak{S}_w(x;t)\in 
\mathbb{Q}[x_i,t_j]_{i,j}=\mathbb{Q}[x_1,x_2,\ldots,t_1,t_2,\ldots]$$
for $w\in S_n$ can be characterized by 
$$\mathfrak{S}_{w_0}(x;t)=\prod_{i+j\leq n}
(x_i-t_j),\qquad
\mathfrak{S}_w(x;t)
=\partial_k\mathfrak{S}_{ws_k}(x;t),
\text{ if }
\ell(ws_k)=\ell(w)+1,$$
where $\partial_k$ is the divided difference operator on polynomials  
$$\partial_kf =\frac{f-f|_{x_k\leftrightarrow x_{k+1}}}{x_k-x_{k+1}}.$$
If we set $t=0$, $\mathfrak{S}_w(x;0)$ is the single Schubert polynomial $\mathfrak{S}_w(x)$.

We also need the \emph{left Demazure operators $\varpi_k$}, see, for example,  \cite[Appendix A]{FGXpuzzle}. For any polynomial $f(x;y)$, let
$$\varpi_k f(x;y)=-\frac{f-f|_{y_k\leftrightarrow y_{k+1}}}{y_k-y_{k+1}}. $$
Then, by setting $\beta=0$ in \cite[Proposition A.1]{FGXpuzzle}, we have 
\begin{equation}\label{leftoperator}
\varpi_k
\mathfrak{S}_{w}(x;y)
=\begin{cases}
\mathfrak{S}_{s_kw}(x;y), & \ell(s_kw)=\ell(w)-1,\\
0,& \ell(s_kw)=\ell(w)+1.
\end{cases}
\end{equation}

\begin{prop}\label{prop:property}
Let $u,v\in S_\infty$. We have 
\begin{align}
\label{eq:1<2<3} 
\supp^1(u,v)& \subseteq 
\supp^2(u,v)\subseteq 
\supp^3(u,v), \\
\label{eq:i->1}
\supp^1(u,v) & = \{w\in \supp^i(u,v):
\ell(w)=\ell(u)+\ell(v)\},\qquad 
\text{for }i=1,2,3. 
\end{align}
\end{prop}
\begin{proof}
The first property follows from  $c_{u,v}^w=c_{u,v}^w(0)$ and 
$c_{u,v}^w(t) = c_{u,v}^w(t;t)$. 
The second property follows from 
$c_{u,v}^w=c_{u,v}^w(t) = c_{u,v}^w(t;y)$ when $\ell(w)=\ell(u)+\ell(v)$ and $c_{u,v}^w=0$ otherwise. 
\end{proof}

It turns out that the triple support is completely determined by the single or double supports.

\begin{theorem}\label{thm:i->3}
For any $u,v\in S_\infty$, 
\begin{equation}
\label{eq:i->3}
\supp^3(u,v) = \bigcup_{\text{ reduced}\,u=u_1\cdot u_2 }\supp^i(u_2,v),\qquad 
\text{for }i=1,2,3. 
\end{equation}
\end{theorem}

\begin{proof}
By \eqref{eq:1<2<3}, it suffices to show ``$\subseteq$'' when $i=1$ and  ``$\supseteq$'' when $i=3$. 

\noindent{\bf The direction ``$\supseteq$''.}
Assume $i=3$. 
Applying $\varpi_k$ to both sides of \eqref{eq:triplecoeff}, by \eqref{leftoperator}, we obtain
$$\varpi_k 
c_{u,v}^w(t;y)
=\begin{cases}
c_{s_ku,v}^w(t;y), & \ell(s_ku)=\ell(u)-1,\\
0, & \ell(s_ku)=\ell(u)+1. 
\end{cases}$$
In particular, 
$\supp^3(s_ku,v)\subseteq \supp^3(u,v)$ if $\ell(s_ku)=\ell(u)-1$. 
If there is a reduced decomposition $u=u_1\cdot u_2$, then, by induction, we have 
\begin{equation}
\supp^3(u_2,v)
\subseteq 
\supp^3(u,v). 
\end{equation}
This proves the inclusion ``$\supseteq$'' of \eqref{eq:i->3} for $i=3$. 

\noindent{\bf The direction ``$\subseteq$''.}
It was conjectured by Samuel \cite{Samuel} and proved by Gao and the third author \cite{Gaox} that the following Graham positivity holds: 
$$c_{u,v}^w(t;y)\in \mathbb{N}[t_i-y_j]_{i,j\geq 0}.$$
Thus we have 
$$c_{u,v}^w(t;y)\neq 0\iff c_{u,v}^w(0;-1)\neq 0.$$
Setting $t=0, y=-1$ in \eqref{eq:triplecoeff}, we get 
$$\mathfrak{S}_{u}(x,-1)\cdot 
\mathfrak{S}_{v}(x)
=\sum_{w\in \supp^3(u,v)} c_{u,v}^w(0,-1)\cdot 
\mathfrak{S}_w(x)
\in \sum_{w\in \supp^3(u,v)} \mathbb{Z}_{>0}\cdot 
\mathfrak{S}_w(x).
$$
By the Cauchy formula  of Schubert polynomials \cite{Mac}, we have
$$\mathfrak{S}_u(x,-1) = \sum_{\text{reduced }u=u_1\cdot u_2}
\mathfrak{S}_{u_1^{-1}}(1)\cdot 
\mathfrak{S}_{u_2}(x).$$
Since the single structure constant $c_{u,v}^w\in \mathbb{Z}_{\ge0}$, 
\begin{align*}
\mathfrak{S}_{u}(x,-1)\cdot 
\mathfrak{S}_{v}(x)
&=\sum_{\text{reduced }u=u_1\cdot u_2}
\mathfrak{S}_{u_1^{-1}}(1)\cdot 
\mathfrak{S}_{u_2}(x)\cdot \mathfrak{S}_{v}(x)\\
&= 
\sum_{\scriptstyle\text{reduced }u=u_1\cdot u_2 \atop \scriptstyle w\in \supp^1(u_2,v)}
\mathfrak{S}_{u_1^{-1}}(1)\cdot c_{u_2,v}^w\cdot
\mathfrak{S}_{w}(x).
\end{align*}
Thus if $w\in\supp^3(u,v)$ then $w\in\supp^1(u_2,v)$ for some reduced decomposition $u=u_1\cdot u_2$, this proves the inclusion ``$\subseteq$'' of \eqref{eq:i->3} for $i=1$. 
\end{proof}

\begin{proof}[Proof of Theorem \ref{thm:123conj'}]
Applying Theorem \ref{thm:i->3} for $i=1$, we have
\[
\supp^2(u,v) =\left(\bigcup_{\text{ reduced}\,u=u_1\cdot u_2 }\supp^1(u_2,v)\right)\cap \left(\bigcup_{\text{ reduced}\,v=v_1\cdot v_2 }\supp^1(u,v_2)\right),
\]
which leads to the equivalence of \eqref{eq:1->2} and \eqref{eq:3->2}. 
The inclusion ``$\subseteq$''  part of  Conjecture \ref{conj:123conj} follows from \eqref{eq:1<2<3}. 
\end{proof}

\section{The Pieri Case}\label{section3}

In this section, we consider the case that $u$ is either 
\begin{equation}
\begin{aligned}
\delta[r,k]&=s_{k-r+1}\cdots s_{k-1}s_k& \quad \text{for some $0\leq r\leq k$, or}\\
\sigma[r,k]&=s_{k+r-1}\cdots s_{k+1}s_{k}& \quad \text{for some }0\leq r,k.
\end{aligned}
\end{equation}

\begin{theorem} \label{th:piericase}
Let $v\in S_\infty.$ Then we have
\begin{align}
\label{eq:epieri}
\supp^2(\delta[r,k],v)&=\supp^3(\delta[r,k],v)\cap \supp^3(v,\delta[r,k])
\\
\label{eq:hpieri}
\supp^2(\sigma[r,k],v)&=\supp^3(\sigma[r,k],v)\cap \supp^3(v,\sigma[r,k]). 
\end{align}
\end{theorem}

We only give a proof of \eqref{eq:epieri}, the proof for \eqref{eq:hpieri} is similar.

Let us first recall the ordinary, double and triple Pieri rules. 
Recall that the $k$-Bruhat order is defined as follows. For two permutations $x,y\in S_\infty$, we say $y$ covers $x$ in the $k$-Bruhat order, denoted as $ x\lessdot_k y,$ if
\[
y=xs_{ab},
\quad
a\leq k<b,
\quad
\ell(y)=\ell(x)+1,
\]
where $s_{ab}$ is the reflection interchanging $a$ and $b$.
Label this edge by the value $x(a)$, and denote
\[
x\xrightarrow{\tau}y
\quad\Longleftrightarrow\quad
x\lessdot_k y
\text{ and }
\tau=x(a).
\]
The $k$-Bruhat order $\le_k$ is the transitive closure of the covering relation $\lessdot_k$.
A  $k$-path $\gamma$ of length $\ell(\gamma)=s$ from $v$ to $w$, denoted as $\gamma:v\rightsquigarrow w$, is a chain
\[
\gamma:
v=x_0
\xrightarrow{\tau_1}x_1
\xrightarrow{\tau_2}\cdots
\xrightarrow{\tau_s}x_s=w.
\]
We say that $\gamma$ is  \emph{decreasing} (resp., \emph{increasing}), if
$
\tau_1>\tau_2>\cdots>\tau_s
$ (resp., $\tau_1<\tau_2<\cdots<\tau_s$).

\begin{theorem}[ordinary Pieri rule \cite{Sottile}]
Let $v\in S_\infty$. We have
\begin{align*}
\mathfrak{S}_{\delta[r,k]}(x)\cdot\mathfrak{S}_v(x)
& =\sum_{\substack{\text{decreasing $k$-path}\\
\gamma:\,v\rightsquigarrow w,\,
\ell(\gamma)=r}}
\mathfrak{S}_{w}(x),&
\mathfrak{S}_{\sigma[r,k]}(x)\cdot \mathfrak{S}_v(x)
& =\sum_{\substack{\text{increasing $k$-path}\\
 \gamma:\,v\rightsquigarrow w,\, 
\ell(\gamma)=r}}
\mathfrak{S}_w(x).   
\end{align*}%
\end{theorem}%
As shown in \cite[Lemma 6]{Sottile} (see also \cite[Lemma 5.2]{FGX}), we have 
\begin{align*}
c^w_{\delta[r,k],v}(0)\neq0
\quad&\Longleftrightarrow\quad
\text{there exists a decreasing $k$-path 
$\gamma:v\rightsquigarrow w$
of length $r$}.\\
&\Longleftrightarrow
\quad w=vs_{a_1b_1}s_{a_2b_2}\cdots s_{a_rb_r}\,\text{with}\,a_1,\ldots,a_r\,\text{distinct, and} \\
&\qquad\quad \,\ell(vs_{a_1b_1}s_{a_2b_2}\cdots s_{a_ib_i})=\ell(v)+i,\,\text{for} \  i=1,\ldots,r.\\
c^w_{\sigma[r,k],v}(0)\neq0
\quad&\Longleftrightarrow\quad
\text{there exists an increasing $k$-path 
$\gamma:v\rightsquigarrow w$
of length $r$}.\\
&\Longleftrightarrow
\quad w=vs_{a_1b_1}s_{a_2b_2}\cdots s_{a_rb_r}\,\text{with}\,b_1,\ldots,b_r\,\text{distinct, and} \\
&\qquad\quad \,\ell(vs_{a_1b_1}s_{a_2b_2}\cdots s_{a_ib_i})=\ell(v)+i,\,\text{for} \ i=1,\ldots,r.
\end{align*}

For two permutations $v$ and $w$, define
\begin{align}
\Delta_k(v,w)&=\{v(i): i\leq k,\,v(i)= w(i)\},\\[5pt]
\Sigma_k(v,w)&=\{v(1),\ldots,v(k)\}\cup \{v(i):  i>k,\, v(i)\neq w(i)\}. 
\end{align}
If $\gamma:v\rightsquigarrow w$
is a decreasing $k$-path with length $r$, then its $r$ edges use $r$ distinct values
of $v(1),\ldots,v(k)$.  Hence
$|\Delta_k(v,w)|=k-r.$
Similarly, if $\gamma:v\rightsquigarrow w$
is an increasing $k$-path with length $r$, then 
$|\Sigma_k(v,w)|=k+r.$
For a set of positive integers $A=\{a_1<a_2<\cdots<a_m\}$ and  any given polynomial $f(x_1,\ldots, x_m)$, we denote by $f(x_A)$   the polynomial    $f(x_{a_1},\ldots, x_{a_m})$.

\begin{theorem}[double Pieri rule \cite{Robinson,LSY}]\label{doublepieri}
Let $v\in S_\infty$. Then
\begin{align*}
\mathfrak{S}_{\delta[r,k]}(x;t)\cdot\mathfrak{S}_v(x;t)
& =
\sum_{\substack{\text{decreasing $k$-path } \gamma:\, v\rightsquigarrow w\\
\delta[r,k]\leq w,\
\ell(\gamma)=r'\le r}}\mathfrak{S}_{\delta[r-r',k-r']}(t_{\Delta_k(v,w)};t)
\cdot\mathfrak{S}_{w}(x;t).\\[5pt]
\mathfrak{S}_{\sigma[r,k]}(x;t)\cdot\mathfrak{S}_v(x;t)
& =
\sum_{\substack{\text{increasing $k$-path }\gamma:\,v\rightsquigarrow w\\
\sigma[r,k]\leq w, \
\ell(\gamma)=r'\le r}}\mathfrak{S}_{\sigma[r-r',k+r']}(t_{\Sigma_k(v,w)};t)
\cdot\mathfrak{S}_{w}(x;t).
\end{align*}
\end{theorem}

\begin{theorem}[triple Pieri rule \cite{Samuel}]\label{triplepieri}
Let $v\in S_\infty$. Then
\begin{align*}
\mathfrak{S}_{\delta[r,k]}(x;y)\cdot\mathfrak{S}_v(x;t)
& =
\sum_{\substack{ \text{decreasing $k$-path}\\
\gamma:\, v\rightsquigarrow w,\,
\ell(\gamma)=r'\le r}}\mathfrak{S}_{\delta[r-r',k-r']}(t_{\Delta_k(v,w)};y)
\cdot\mathfrak{S}_{w}(x;t).\\[5pt]
\mathfrak{S}_{\sigma[r,k]}(x;y)\cdot\mathfrak{S}_v(x;t)
& =
\sum_{\substack{\text{increasing $k$-path}\\
\gamma:\,v\rightsquigarrow w,\,
\ell(\gamma)=r'\le r}}\mathfrak{S}_{\sigma[r-r',k+r']}(t_{\Sigma_k(v,w)};y)
\cdot\mathfrak{S}_{w}(x;t).
\end{align*}
\end{theorem}

\begin{example}
\def\O#1{\raisebox{0.1ex}{\makebox[1pc]{%
    \makebox[0pc]{$\bigcirc$}%
    \makebox[0pc]{$\scriptstyle #1$}}}}
Let $u=231=s_1s_2$ and $v=312=s_2s_1$. 
We have $u=\delta[r,k]$ for $k=2, $ $r=2$. 
The $2$-Bruhat order over $v=312$ looks like 
$$
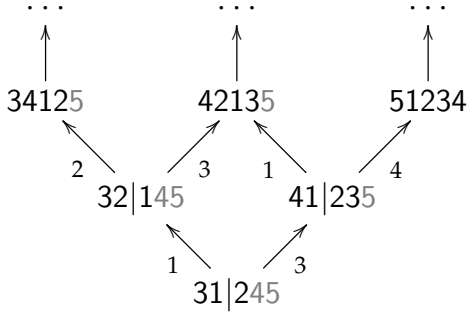
\begin{matrix}
\def\objectstyle{\sf}
\xymatrix@!=1.0pc{
\cdots&&\cdots&&\cdots\\
3412\color{gray}5\ar[u]&&
4213\color{gray}5\ar[u]&& 
51234\color{gray}\ar[u]\\&
32|1\color{gray}45
    \ar[ul]^2\ar[ur]_3&&
41|23\color{gray}5
    \ar[ul]^1\ar[ur]_4\\&&
31|2\color{gray}45
    \ar[ul]^1\ar[ur]_3
}
\end{matrix}\qquad 
\begin{array}{cll}\hline
\vphantom{\dfrac23}&
\hfill\textnormal{decreasing paths}\hfill & \\\hline
\O{1}&{\sf 31|2\color{gray}45} 
& \not\geq {\sf 23|1\color{gray}45}\\
\O{2}&
{\sf 31|2\color{gray}45}
\stackrel{1}\longrightarrow
{\sf 32|1\color{gray}45} 
&\geq {\sf 23|1\color{gray}45}\\
\O{3}&
{\sf 31|2\color{gray}45}
\stackrel{3}\longrightarrow
{\sf 41|23\color{gray}5} 
&\not\geq {\sf 23|1\color{gray}45}\\
\O{4}&
{\sf 31|2\color{gray}45}
\stackrel{3}\longrightarrow
{\sf 41|23\color{gray}5}
\stackrel{1}\longrightarrow
{\sf 42|13\color{gray}5} 
&\geq {\sf 23|1\color{gray}45}\\\hline
\end{array}
$$
By Theorem \ref{triplepieri}, the triple and double coefficients $c_{u,v}^w(t;y)$ and $c_{u,v}^w(t)$ can be computed as follows. 
$$\begin{array}{c|c|c|c|c}\hline
\vphantom{\dfrac23}
&
\textnormal{length $r'$}&
\Delta_k&
c_{u,v}^w(t;y)=\mathfrak{S}_{\delta[r-r',k-r']}(t_{\Delta_{k}};y)& 
c_{u,v}^w(t;t)=
\mathfrak{S}_{\delta[r-r',k-r']}(t_{\Delta_{k}};t)\\\hline
\O{1} &0  & \{1,3\} & (t_1-y_1)(t_3-y_1) & 0 \\
\O{2} &1& \{3\}& t_3-y_1 & t_3-t_1\\
\O{3}
    &1& \{1\} & t_1-y_1 & 0 \\
\O{4}    &2& \varnothing & 1 & 1
\\\hline
\end{array}$$
This agrees with the red circle in Example \ref{eg:runningeg}. 

Similarly, 
let $u=312=s_2s_1$ and $v=231=s_1s_2$. 
We can view $u=\sigma[r,k]$ for $k=1,r=2$. 
By Theorem \ref{triplepieri}, the triple and double coefficients $c_{u,v}^w(t;y)$ and $c_{u,v}^w(t)$ can be computed as follows
$$
\hspace{-10pc}
\begin{array}{cc|c|c|c|c}\hline
\vphantom{\dfrac12}
\textnormal{increasing paths}&&
\textnormal{length }r'&\Sigma_k&
c_{u,v}^w(t;y)
& 
c_{u,v}^w(t)
\\\hline
{\sf 2|31\color{gray}45}\hfill&
    \not\geq{\sf 3|12\color{gray}45}&
    0& \{2\} & (t_2-y_1)(t_2-y_2) & 0\\
{\sf 2|31\color{gray}45}
\stackrel{2}\to
{\sf 3|21\color{gray}45}\hfill&
    \geq{\sf 3|12\color{gray}45}&
    1 & \{2,3\} & t_2+t_3-y_1-y_2 & t_3-t_1\\
{\sf 2|31\color{gray}45}
\stackrel{2}\to
{\sf 3|21\color{gray}45}
\stackrel{3}\to
{\sf 4|213\color{gray}5}\hfill&
    \geq{\sf 3|12\color{gray}45}& 
    2 & \{2,3,4\} & 1 &1\\\hline
\end{array}
\hspace{-10pc}
$$
This agrees with the blue circle in Example \ref{eg:runningeg}. 
\end{example}

\begin{lemma}\label{fangda1}
For any $u,v\in S_\infty$, 
$$\supp^i(u,v)\subseteq 
\{w\in S_\infty: v\leq w\},\qquad 
\text{for }i=1,2,3. $$
\end{lemma}

\begin{proof} 
By \eqref{eq:1<2<3}, it suffices to prove the case $i=3$. In fact, we can replace $\mathfrak{S}_u(x;y)$ by any polynomial  $f(x;t;y)$, and show that in the expansion 
\begin{equation}\label{localiz}
f(x;t;y)\cdot \mathfrak{S}_v(x;t)
=\sum_{w} d_w(t;y)\cdot \mathfrak{S}_{w}(x;t),   
\end{equation}
each non-vanishing $\mathfrak{S}_{w}(x;t)$ must satisfy $v\le w.$

Recall that the localization of any Schubert polynomial $\mathfrak{S}_z(x;t)$ at a fixed point $w=w(1)\cdots w(n)\in S_n$ is 
\[
\mathfrak{S}_z(x;t)|_w:=\mathfrak{S}_z(wt;t)=\mathfrak{S}_z(t_{w(1)},t_{w(2)},\ldots,t_{w(n)};t).
\]
It is well known that
\[
\mathfrak{S}_{z}(x;t)|_w=\mathfrak{S}_{z}(wt;t)\neq 0 \quad \text{if and only if \quad} z\le w, 
\]
see, for example, Billey \cite{Billey} and Knutson \cite{Knutson}.
 
Now suppose to the contrary that there is some $\overline{w}$ appearing in the right-hand side of \eqref{localiz} such that $v\not\le \overline{w}$. Choose such a $\overline{w}$  with minimal length. Then localize both sides of \eqref{localiz} at $\overline{w}$, we obtain
\begin{equation}\label{eq:local2}
f(\overline{w}t;t;y)\cdot \mathfrak{S}_v(\overline{w}t;t)
=\sum_{w} d_w(t;y)\cdot \mathfrak{S}_{w}(\overline{w}t;t).
\end{equation}
Since $\mathfrak{S}_{z}(x;t)|_w=\mathfrak{S}_{z}(wt;t)=0$, if $w\not\ge z$, and $\mathfrak{S}_{w}(x;t)|_w=\mathfrak{S}_{w}(wt;t)\neq0,$ we find that  \eqref{eq:local2} becomes
\begin{equation}\label{eq:local3}
0=d_{\overline{w}}(t;y)\mathfrak{S}_{\overline{w}}(\overline{w}t;t)+\sum_{z<\overline{w}}d_z(t;y)\mathfrak{S}_z(\overline{w}t;t).     
\end{equation}
Since $\overline{w}$  has minimal length, each $\mathfrak{S}_z(\overline{w}t;t)$ in the  right-hand side sum  of \eqref{eq:local3} satisfies $v\le z<\overline{w}$, which contradicts with the hypothesis $v\not\le \overline{w}$. Thus the sum in the right-hand side of \eqref{eq:local3} is empty, and we are led to $0=d_{\overline{w}}(t;y)\mathfrak{S}_{\overline{w}}(\overline{w}t;t)$,  contradiction.
\end{proof}

\begin{proof}[Proof of Theorem \ref{th:piericase}]

By Lemma \ref{fangda1}, 
\begin{align*}
\supp^3(\delta[r,k],v)\cap \supp^3(v,\delta[r,k])&\subseteq \supp^3(\delta[r,k],v)\cap
\{w\in S_\infty: \delta[r,k]\leq w\}\nonumber\\
&=\big\{w\in \supp^3(\delta[r,k],v): \delta[r,k]\le w\big\}.
\end{align*}
On the other hand, 
\[
\supp^2(\delta[r,k],v)=\big\{w\in \supp^3(\delta[r,k],v): c_{\delta[r,k],v}^w(t;t)\neq 0\big\}.
\]
Therefore, in order to prove \eqref{eq:epieri}, it suffices to show
\begin{equation}\label{incls} 
\big\{w\in \supp^3(\delta[r,k],v): \delta[r,k]\le w\big\}\subseteq \big\{w\in \supp^3(\delta[r,k],v): c_{\delta[r,k],v}^w(t;t)\neq 0\big\}.
\end{equation}

By Theorem \ref{triplepieri}, $c_{\delta[r,k],v}^w(t;y)=\mathfrak{S}_{\delta[r-r',k-r']}(t_{\Delta_k(v,w)};y)$. Thus
\[
c_{\delta[r,k],v}^w(t;t)=\mathfrak{S}_{\delta[r-r',k-r']}(t_{\Delta_k(v,w)};t)=\mathfrak{S}_{\delta[r-r',k-r']}(x;t)|_{\pi},
\]
where $\pi\in S_\infty$ is  any permutation such that $\pi([k-r'])=\Delta_k(v,w).$ Therefore,
\begin{align*}
c_{\delta[r,k],v}^w(t;t)\neq 0 &\Longleftrightarrow \mathfrak{S}_{\delta[r-r',k-r']}(x;t)|_{\pi}\neq0\\
&\Longleftrightarrow \delta[r-r',k-r']\le \pi, \,\text{for any permutation $\pi([k-r'])=\Delta_k(v,w)$}.
\end{align*}
Without loss of generality, we can assume $\Delta_k(v,w)=\{\delta_1<\cdots<\delta_{k-r'}\}$, and let $\pi\in S_\infty$ be the $(k-r')$-Grassmannian permutation determined by
\(
\pi(i)=\delta_i    
\),
for $i\le k-r'$. 
To show the inclusion \eqref{incls} holds, it suffices to show that 
\[
\text{if $\delta[r,k]\le w$, then  $\delta[r-r',k-r']\le \pi$ }
\]
for $\pi(i)=\delta_i$ as defined above.

Recall that for two subsets $A,B\subseteq[n]$ with the same cadinality, we say $A\le B$ under the Gale order if  the $i$-th smallest element of $A$ is less than or equal to the $i$-th smallest element of $B$ for all $i$. It is easy to see that if $A\le B$, then 
\begin{equation}\label{eq:galeorder}
A\setminus \{\max(A)\}\le B\setminus \{b\}  
\end{equation}
for any $b\in B.$
For two permutations $u,v\in S_n$, we have $u\le v$ under Bruhat order if and only if $u[i]:=\{u(1),\ldots,u(i)\}\le \{v(1),\ldots,v(i)\}=:v[i],$ for all $1\le i\le n$.
Moreover, if $u$ is a $k$-Grassmannian permutation, then $u\leq v$ if and only if $u[i]\leq v[i]$ for $i=1,\ldots,k$. 

Now we prove the inclusion \eqref{incls}. Assume that $\delta[r,k]=s_{k-r+1}\cdots s_k\le w$. Then in the Gale order, we have
\begin{equation}\label{gale}
\{1,2,\ldots,k-r,k-r+2,\ldots,k+1\}\le\{w(1),\ldots,w(k)\}.
\end{equation}
Deleting the largest $r'$ elements from the left-hand side of \eqref{gale}, and deleting the $r'$ labels of the decreasing $k$-path $\gamma: v\rightsquigarrow w$ from the right-hand side of \eqref{gale}, by \eqref{eq:galeorder}, we are led to 
\begin{equation}\label{galeconclu}
\{1,2,\ldots,k-r,k-r+2,\ldots,k-r'+1\}\le \Delta_k(v,w).
\end{equation}
Since $\delta[r-r',k-r']$ is a Grassmannian with descent at $k-r'$, the first $k-r'$ elements are increasing, we see that  \eqref{galeconclu}  implies $\delta[r-r',k-r']\le \pi$, as desired. 
\end{proof}

\section{Separated Descents Case}\label{section4}

In this section, we consider the case $u,v$ are of separated descents at position $k$, i.e., 
$$\max(\operatorname{des}(u))\leq k\leq \min(\operatorname{des}(v)),$$
where for a permutation $u\in S_n$, its descent set is $des(u)=\{i: u(i)>u(i+1)\}$.

\begin{theorem}\label{th:4.1}
Let $u,v$ be permutations with separated descents. Then
\begin{align}\label{eq:grass}
\supp^2(u,v)=\supp^3(u,v)\cap \supp^3(v,u).
\end{align}
\end{theorem}

We first prove the following lemma.
\begin{lemma}\label{lemma4.1}
For $u,v$ with $\max(\operatorname{des}(u))\leq k\le\min(\operatorname{des}(v))$, we have
    $$\supp^3(v,u)\subseteq \{w\in S_\infty: u(i)\leq w(i),\text{ for  }1\leq i\leq k\}.$$
\end{lemma}
\begin{proof}
We now consider the expansion
\[
\mathfrak{S}_v(x;y)\cdot \mathfrak{S}_{u}(x;t)
=\sum_{w\in S_\infty} c_{v,u}^w(t;y)
\cdot\mathfrak{S}_{w}(x;t). 
\]
Since $\min(\operatorname{des}(v))\ge k$, 
$\mathfrak{S}_v(x;y)$ is symmetric in $x_1,\ldots,x_k$. Hence it belongs to the $\mathbb{Q}[y]$-algebra generated by 
\begin{equation}\label{eq:itpieri}
e_1(x_1,\ldots,x_k),\ldots,
e_k(x_1,\ldots,x_k),\quad x_1+\cdots+x_{k+1}, \quad x_1+\cdots+x_{k+2},
\ldots.
\end{equation}
We claim that in the expansion of any Schubert polynomial $\mathfrak{S}_z(x;t)$ multiplied with any generator in \eqref{eq:itpieri}, the appearing $\mathfrak{S}_w(x;t)$ must satisfy: $z(i)\le w(i)$ for $i\le k$.

For  $1\le r\le k$, we have 
$\mathfrak{S}_{\delta[r,k]}(x;0)=e_r(x_{1},\ldots,x_{k}).$ Let $y=0$ in  the triple Pieri rule Theorem \ref{triplepieri}, we find that if $\mathfrak{S}_w(x;t)$ appears with nonzero coefficient in 
$\mathfrak{S}_z(x;t)\cdot e_r(x_{1},\ldots,x_{k})$,
then there is a decreasing $k$-path,
\[
    z=z_0\lessdot_k z_1\lessdot_k \cdots \lessdot_k z_s=w.
\]
At each step, we have $z_{j+1}=z_js_{ab}$, for some $a\le k<b$.  Because $z_{j}\lessdot z_{j+1}$, we have $z_{j}(a)<z_{j}(b)=z_{j+1}(a)$, while all other entries among the first $k$ positions remain unchanged.

Similarly, for any $m>k$, we have $\mathfrak{S}_{\delta[1,m]}(x;0)=e_1(x_{1},\ldots,x_{m})=x_{1}+\cdots+x_{m}$. The triple Pieri rule Theorem \ref{triplepieri} implies that the appearing  $\mathfrak{S}_w(x;t)$ in the expansion  $\mathfrak{S}_z(x;t)\cdot (x_1+\cdots+x_m)$  are  those $w$ such that there are decreasing $m$-paths 
\[
z=z_0
\lessdot_m z_1
\lessdot_m\cdots
\lessdot_m z_s=w,
\]
with each step has the form
$$z_{j+1}=z_js_{ab},\quad a\le m<b.$$
If $a> k$, then the first $k$ entries are unchanged. 
If $a\leq k$, then $z_{j+1}(a)=z_j(b)>z_j(a)$, and the remaining entries in the first $k$ positions are unchanged.   

It is easy to see that in each case $z(i)\le w(i)$ for $1\le i\le k$. Now express $\mathfrak{S}_v(x;y)$ as a polynomial in the generators in \eqref{eq:itpieri}. Starting with $\mathfrak{S}_{u}(x;t)$ and multiplying successively by these generators, we can eventually conclude that
$u(i)\le w(i)$ for $1\le i\le k$.
\end{proof}

Therefore, by Lemma \ref{lemma4.1}, we have  the following inclusions
\begin{align*}
\supp^2(u,v)&\subseteq \supp^3(u,v)\cap  \supp^3(v,u)\\
&\subseteq \supp^3(u,v)\cap\{w\in S_\infty: u(i)\leq w(i)\text{ for  }1\leq i\leq k\}.
\end{align*}
To prove Theorem \ref{th:4.1}, it suffices to show the following theorem.

\begin{theorem}\label{th:puzzle}
For $u,v$ with $\max(\operatorname{des}(u))\leq k\le\min(\operatorname{des}(v))$, we have
$$\supp^2(u,v)=\supp^3(u,v)\cap \{w\in S_\infty: u(i)\leq w(i),\text{ for  }1\leq i\leq k\}.$$
\end{theorem}

To prove Theorem \ref{th:puzzle}, we utilize the pipe puzzle model for $c_{u,v}^w(t;y)$ given in \cite[Theorem 2.2]{FGXpuzzle}.

Consider an $n\times n$ grid with boundaries  labeled  as follows.  
\begin{align}\label{eq:boarduvw}
\BPD[1.5pc]{
\M{}\M{\theta_{v}^1}\M{\theta_v^2}\M{\cdots}\M{\cdots}\M{\theta_v^n}\\
\M{0}\O\O\M{\cdots}\M{\cdots}\O\M{\kappa_{u}^1}\\
\M{0}\O\O\M{\cdots}\M{\cdots}\O\M{\kappa_{u}^2}\\
\M{\vdots}\M{\vdots}\M{\vdots}\M{\ddots}\M{\ddots}\M{\vdots}\M{\vdots}\\
\M{\vdots}\M{\vdots}\M{\vdots}\M{\ddots}\M{\ddots}\M{\vdots}\M{\vdots}\\
\M{0}\O\O\M{\cdots}\M{\cdots}\O\M{\kappa_{u}^n}\\
\M{}\M{\eta_w^1}\M{\eta_w^2}\M{\cdots}\M{\cdots}\M{\eta_w^n}
}
\qquad 
\begin{array}{r@{\,}l}
\kappa_u^i &=\begin{cases}
    u^{-1}(i), & u^{-1}(i) \leq k,\\[5pt]
    0, & u^{-1}(i)>k.
    \end{cases}\\[4ex]
\theta_v^i &=\begin{cases}
    0, & v^{-1}(i)\leq k,\\[5pt]
    v^{-1}(i), & v^{-1}(i) >k.
    \end{cases}\\[4ex]
\eta_w^i & = w^{-1}(i).
\end{array}
\end{align}

Since $\max(des(u))\le k\le \min(des(v))$,   the nonzero labels on 
the right boundary are $1, \ldots, k$, and   the nonzero labels on the top boundary are $k+1,\ldots, n$. It is easy to see that  we can reconstruct $u,v,w$ from the boundary labelings.  We omit the labels 0 on the boundaries. 

Now tile the $n\times n$ grid  with  boundaries labeled as in \eqref{eq:boarduvw} with the following  admissible tiles   
\begin{align}\label{eq:Schtile}
\begin{array}{cccccc}
\BPD[1.5pc]{\O}&
\BPD[1.5pc]{\X}&
\BPD[1.5pc]{\F}&
\BPD[1.5pc]{\J}&
\BPD[1.5pc]{\I}&
\BPD[1.5pc]{\H}\\
\end{array}
\end{align}
The curves drawn  on the tiles are referred to as \emph{pipes}. 
A tiling of \eqref{eq:boarduvw} built upon the tiles in \eqref{eq:Schtile} is a network of pipes such that 
\begin{enumerate}
    \item There are a total of $n$ pipes, among which $k$ pipes   enter horizontally from rows   on the right boundary labeled $1,\ldots, k$, and $n-k$ pipes enter vertically  from   columns  on the top boundary labeled $k+1,\ldots, n$. 
    The  pipes inherit the labels of the corresponding rows and columns. 

    \item The $n$ pipes exit vertically  through the bottom boundary.   No two pipes cross twice or more. For $1\le i\le n$, we call the pipe exiting through the $i$-th column of the bottom boundary the $i$-th pipe. 
\end{enumerate}

A \emph{Schubert pipe puzzle} for $u,v,w$ is a tiling of \eqref{eq:boarduvw} with the tiles in \eqref{eq:Schtile}, subject to the following restriction on the $\BPD{\X}$  tiles: The label of the horizontal pipe in $\BPD{\X}$ must be smaller than the label of the vertical pipe.
For example, 
    $$
    \BPD{\M{3}\M{}\M{}\\
    \I\O\O\\
    \I\O\F\M{2}\\
    \I\F\X\M{1}\\
    \M{3}\M{1}\M{2}}\text{ is allowed, while }
    \BPD{\M{3}\M{}\M{}\\
    \I\O\O\\
    \I\O\F\M{1}\\
    \I\F\BPDfr{\color{red}\XX}\M{2}\\
    \M{3}\M{2}\M{1}}\text{ is not allowed.}
$$
Denote by $\PP(u,v,w)$ the set of Schubert pipe puzzles for $u,v,w$. 
For each $\pi \in \PP(u,v,w)$, define its \emph{weight} by
$$\wt(\pi; t, y)=\prod_{(i,j)} (t_j-y_i),$$
where the sum is over empty tiles $\BPD{\O}$ at  $(i,j)$-position (in the matrix coordinate).  If we set $y_i=t_i$, then a pipe puzzle has weight zero if and only if there exists at least one empty tile $\BPD{\O}$ on the diagonal. 

\begin{theorem}[\cite{FGXpuzzle}]\label{HHUU}
Let  $u,v \in S_\infty$ be permutations with separated descents at position $k$.
For $w\in S_\infty$, we have
\begin{equation*} 
    c_{u, v}^w(t;y) = \sum_{\pi\in \PP(u,v,w)} \wt(\pi; t,y).
\end{equation*} 
\end{theorem}

\begin{example}
For our running example $u=312$ and $v=231$, we can take $k=1$ such that
$$\max(\operatorname{des}(u))=1\leq k\leq \min(\operatorname{des}(v))=2.$$
Label the top boundary by $v^{-1}=3{\color{gray}1}24$ and label the right boundary by $u^{-1}={\color{gray}23}1{\color{gray}4}$.  We have the following pipe puzzles, which implies $\supp^3(u,v)=\{231,321,4213\}$.
$$
\begin{matrix}
 \BPD{\M{3}\M{}\M{2}\M{4}\\\I\O\I\I\\\I\O\I\I\\\I\F\X\X\M{1}\\\I\I\I\I\\\M{3}\M{1}\M{2}\M{4}}
 &
 \BPD{\M{3}\M{}\M{2}\M{4}\\\I\F\J\I\\\I\I\O\I\\\I\I\F\X\M{1}\\\I\I\I\I\\\M{3}\M{2}\M{1}\M{4}}
\BPD{\M{3}\M{}\M{2}\M{4}\\\I\O\I\I\\\I\F\J\I\\\I\I\F\X\M{1}\\\I\I\I\I\\\M{3}\M{2}\M{1}\M{4}}
 
 &
 \BPD{\M{3}\M{}\M{2}\M{4}\\\I\F\J\I\\\I\I\F\J\\\I\I\I\F\M{1}\\\I\I\I\I\\\M{3}\M{2}\M{4}\M{1}}\\
c_{u,v}^{2314}(t;y)=(t_2-y_1)(t_2-y_2)
&
c_{u,v}^{3214}(t;y)=(t_3-y_2)+(t_2-y_1)
&
c_{u,v}^{4213}(t;y)=1
\end{matrix}
$$
If we use $k=2$, we get the same result.
$$
\begin{matrix}
 \BPD{\M{3}\M{}\M{}\M{4}\\\I\O\F\X\M{2}\\\I\O\I\I\\\I\F\X\X\M{1}\\\I\I\I\I\\\M{3}\M{1}\M{2}\M{4}}
 &
 \BPD{\M{3}\M{}\M{}\M{4}\\\I\F\H\X\M{2}\\\I\I\O\I\\\I\I\F\X\M{1}\\\I\I\I\I\\\M{3}\M{2}\M{1}\M{4}}
\BPD{\M{3}\M{}\M{}\M{4}\\\I\O\F\X\M{2}\\\I\F\J\I\\\I\I\F\X\M{1}\\\I\I\I\I\\\M{3}\M{2}\M{1}\M{4}}
 
 &
 \BPD{\M{3}\M{}\M{}\M{4}\\\I\F\H\X\M{2}\\\I\I\F\J\\\I\I\I\F\M{1}\\\I\I\I\I\\\M{3}\M{2}\M{4}\M{1}}\\
c_{u,v}^{2314}(t;y)=(t_2-y_1)(t_2-y_2)
&
c_{u,v}^{3214}(t;y)=(t_3-y_2)+(t_2-y_1)
&
c_{u,v}^{4213}(t;y)=1
\end{matrix}
$$
This agrees with the blue circle in Example \ref{eg:runningeg}. 
\end{example}

\bigbreak
\begin{proof}[Proof of Theorem \ref{th:puzzle}]
We only need to show the inclusion "$\supseteq$". That is, suppose that $w\in \supp^3(u,v)$ with $u(l)\leq w(l)$ for all $1\leq l\leq k$. 
We aim to show that there exists a pipe puzzle in $\operatorname{PP}(u,v,w)$  with no empty tile $\BPD{\O}$ on the diagonal.  

 We first claim that the $i$-th pipe of the pipe puzzle in $\operatorname{PP}(u,v,w)$ must pass through the $i$-th row. 
Let $j=w^{-1}(i)$. If $j>k$, then the $i$-th pipe (labeled by $j$) enters vertically from the top boundary in column $v(j)$. Obviously, the $i$-th pipe passes through the $i$-th row.
If $j\leq k$, then the $i$-th pipe enters horizontally from the right boundary in row $u(j)$.  By our assumption  $u(l)\le w(l)$ for all $l\le k$, since $j\le k$, we have $u(j)\leq w(j)=i$, and so the $i$-th pipe enters at row $u(j)\le i$. Thus the $i$-th pipe must pass through the $i$-th row.  The claim follows.

Now pick a pipe puzzle  $\pi\in\operatorname{PP}(u,v,w)$. Assume that $\pi$ has an empty tile $\BPD{\O}$ at position $(i,i)$ with $i$ smallest. It will be clear that the proof of the case $i=1$ is contained in the proof of the case $i>1$. Thus we assume $i>1.$  Since the first pipe must pass through the first row, and the $(1,1)$-position is not an empty tile $\BPD{\O}$, we see that the first pipe must go directly upward to the $(1,1)$-position. Similarly, the second pipe must go directly upward to the $(2,2)$-position. Thus we can conclude that for all $i_0<i$, the $i_0$-th pipe goes directly upward to the $(i_0,i_0)$-position.  
 
Therefore,  there are no $\BPD{\H}$ tiles, $\BPD{\J}$ tiles and $\BPD{\X}$ tiles directly below the $(i,i)$-position. The $i$-th pipe runs as follows from the bottom to the $i$-th row: 
$$
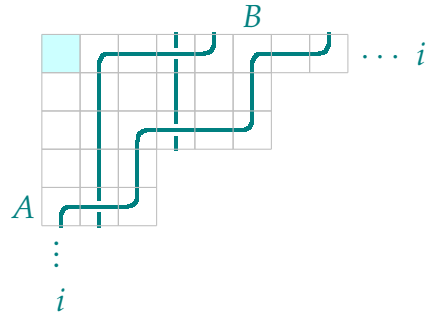
\def\o{\BPDfr{\put(0,0){\color{lightcyan}\rule{\unitlength}{\unitlength}}}}\BPD{
\M{}\M{}\M{}\M{}\M{}\M{}\M{B}\M{}\\
\M{}\o\F\H\X\J\F\H\J\M{\quad \cdots}\M{\quad i}\\
\M{}\O\I\O\I\O\I\\
\M{}\O\I\F\X\H\J\\
\M{}\O\I\I\\
\M{A}\F\X\J\\
\M{}\,\, \vdots\\[5pt]
\M{}\M{i}}
$$

Trace the $i$-th pipe upward from the bottom boundary of column $i$.
Let $A$ be the  tile where it leaves column $i$, and let $B$ be the  tile where it enters row $i$.
We consider the region bounded by the portion of the $i$-th pipe from $A$ to $B$, together with the corresponding segments of column $i$ and row $i$ meeting at the $(i,i)$-position.
All subsequent operations are performed inside this region. We aim to perform a "general droop" operation in this region as follows. Since no pipes enter from the left of this region and no pipes can cross twice or more, all pipes must enter this region from the bottom and leave at the top. 
\begin{itemize}
\item Erase the segment of the $i$-th pipe from tile $A$ to tile $B$. 
\item Apply the droop operation in this region. Since all pipes exit this region on the top, for each tile $\BPD{\F}$,  there is a $\BPD{\J}$ 
 tile  to  the right of it. We choose the first $\BPD{\O}$  in the same column below each $\BPD{\J}$ to perform the droop operation. Locally, this operation is shown as follows:    $$\BPD{\F\H\H\J\\\I\O\O\O}\rightarrow\BPD{\O\O\O\I\\\F\H\H\J}$$
 \item Draw horizontal and vertical pipes from $(i,i)$-position until they reach the tiles $A$ and $B$. 
\end{itemize}

For example, in the right pipe puzzle below, the tile at $(i,i)$ and all tiles directly below it contain no $\BPD{\O}$.   
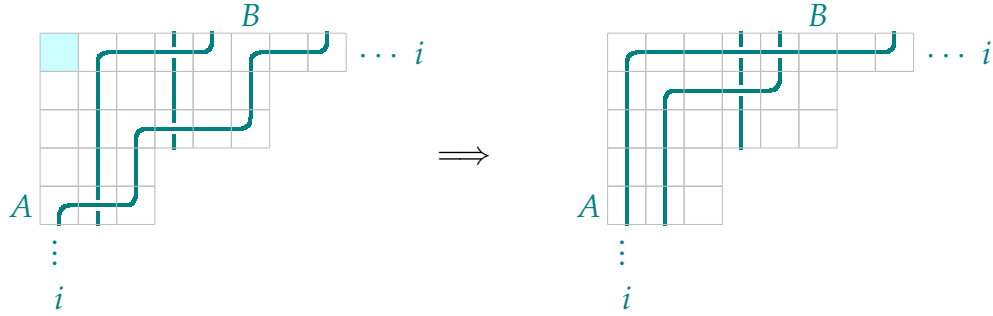
$$\def\o{\BPDfr{\put(0,0){\color{lightcyan}\rule{\unitlength}{\unitlength}}}}\BPD{
\M{}\M{}\M{}\M{}\M{}\M{}\M{B}\M{}\\
\M{}\o\F\H\X\J\F\H\J\M{\quad \cdots}\M{\quad i}\\
\M{}\O\I\O\I\O\I\\
\M{}\O\I\F\X\H\J\\
\M{}\O\I\I\\
\M{A}\F\X\J\\
\M{}\,\, \vdots\\[5pt]
\M{}\M{i}}
\Longrightarrow\qquad
\BPD{
\M{}\M{}\M{}\M{}\M{}\M{}\M{B}\M{}\\
\M{}\F\H\H\X\X\H\H\J\M{\quad \cdots}\M{\quad i}\\
\M{}\I\F\H\X\J\O\\
\M{}\I\I\O\I\O\O\\
\M{}\I\I\O\\
\M{A}\I\I\O\\
\M{}\,\, \vdots\\[5pt]
\M{}\M{i}}
$$
Applying this operation recursively. Since the $(n,n)$-position can not be an empty tile $\BPD{\O}$, we can obtain a pipe puzzle with no $\BPD{\O}$ on the diagonal. Thus $w\in \supp^2(u,v)$. 
\end{proof}

\section{Inverse Grassmannians}\label{section5}

In this section, we consider the case $u$ and $v$ are inverse Grassmannians; that is,  both  $u^{-1}$ and $v^{-1}$ have only
one descent.  
  
\begin{theorem}\label{thm:inverse grass cases}
Let $u,v$ be two inverse Grassmannian permutations. Then we have
\[
\supp^2(u,v)=\supp^3(u,v)\cap \supp^3(v,u).
\]
\end{theorem}

We use the characterizations of the double and triple Schubert structure constants in \cite{S-orbit}. Let us start with  the descriptions of $\supp^3(u,v)$ and $\supp^3(v,u)$. 

\subsection{Preclans}\label{sec:preclans}
A \emph{preclan} is a partial matching whose unmatched nodes are either colored by $\clan{+}$ or $\clan{-}$, or left uncolored.
We say a preclan is a $(p,m,q)$-preclan if 
\begin{itemize}
    \item the number of $\clan{+}$'s and matchings is $p$; 
    \item the number of $\clan{-}$'s and matchings is $q$; 
    \item the number of uncolored nodes $\clan{,}$'s is $m$. 
\end{itemize}
We will always denote $n=p+m+q$. 
 
For example, the following diagram is a $(4,2,5)$-preclan 
\begin{equation}\label{eq:preclaneg}
\gamma=\clan{6-84,-..,+.}.
\end{equation}
We associate a $p$-inverse Grassmannian permutation $v_{\gamma}$ and a $q$-inverse Grassmannian $u_\gamma$ to $\gamma$ as follows.
\begin{itemize}
    \item 
To define $v_\gamma$,  first label the left-ends and $\clan{+}$'s of $\gamma$ from left to right with $1,2,\ldots,p$, and then label the other nodes by $p+1,\ldots,n$ from left to right. Then $v_\gamma$ is obtained by reading the labels of $\gamma$ from left to right.

\item Similarly, to define $u_\gamma$,  first label the left-ends and $\clan{-}$'s of $\gamma$ from left to right with $1,2,\ldots,q$, and then label the other nodes by $q+1,\ldots,n$ from left to right. Then $u_\gamma$ is obtained by reading the labels of $\gamma$ from left to right.
\end{itemize}
For example, let $\gamma$ be the preclan in \eqref{eq:preclaneg}. We have 
\begin{equation}\label{uvgamma}
\def\m#1{\makebox[1.2pc]{$#1$}}
\begin{aligned}
\gamma&=\,\clan{6-84,-..,+.}\\
v_\gamma &= \,\m{1}\m{5}\m{2}\m{3}\m{6}\m{7}\m{8}\m{9}\m{10}\m{4}\m{11}\\ 
u_\gamma &= \,\m{1}\m{2}\m{3}\m{4}\m{6}\m{5}\m{7}\m{8}\m{9}\m{10}\m{11}.\\
\end{aligned}
\end{equation}

A preclan $\gamma$ also determines a bounded affine permutation $f_{\gamma}\in \tilde{S}_{p+m+q}$, see  \cite[Definition 6.1]{S-orbit}. 
We use
$\widetilde{F}_{f_{\gamma}}$ to  denote the corresponding affine Stanley symmetric
function. Since the explicit expressions of these two objects are not needed here, we omit their detailed definitions. 

 \subsection{Weak order on preclan}
 
For a preclan $\gamma$, 
we define $s_k*\gamma$ to be the preclan obtained from $\gamma$ by the following 9 local moves on the $k$-th node and the $(k+1)$-st node. Otherwise, let $s_k*\gamma=\gamma$.

\def\mpto{\rotatebox[origin=c]{-90}{$\mapsto$}}
\begin{gather}\label{eq:clanweak1}
\begin{matrix}
\begin{matrix}
\clan{\dots+-\dots}\\
\clan{\dots-+\dots}
\end{matrix}&\,\,&
\clan{\dots\pm2\dots.}&\,\,&
\clan{2\dots.\pm\dots}\\[-1ex]
\mpto&&\qquad\quad\mpto\hfill&&\hfill\mpto\quad\qquad\\
\clan{\dots1.\dots}&&
\clan{\dots3\pm\dots.}&&
\clan{3\dots\pm.\dots}
\end{matrix}\\[1ex]\label{eq:clanweak2}
\begin{matrix}
\clan{2\dots.2\dots.}&
\clan{34\dots.\dots.}&
\clan{4\dots3\dots..}\\[-1ex]
\mpto&\quad\mpto\hfill&\hfill\mpto\quad\\
\clan{3\dots3.\dots.}&
\clan{52\dots.\dots.}&
\clan{5\dots2\dots..}
\end{matrix}\\[1ex]\label{eq:clanweak3}
\quad
\begin{matrix}
\clan{,\pm}&\qquad&
\clan{,2\dots.}& \quad&
\clan{3\dots,.}
\\
\mpto&& \mpto \qquad\quad && \qquad\mpto\\
\clan{\pm,}&& \clan{3,\dots.} && 
\clan{2\dots.,}
\end{matrix}
\end{gather}

The weak order on all $(p,m,q)$-preclans is generated by the relations $\gamma\leq s_k*\gamma$.
For a permutation $w\in S_n$ and a $(p,m,q)$-preclan $\gamma$, we define 
$$w*\gamma = s_{i_1}*\cdots *s_{i_\ell}*\gamma$$
for any reduced decomposition $w=s_{i_1}\cdots s_{i_\ell}$. In fact, $w*\gamma$ does not depend on the choices of  reduced decompositions of $w$. 

\begin{example}
When $p=q=m=1$, the Hasse diagram of the weak order of all $(1,1,1)$-preclans is displayed in the following.

$$\xymatrix@C=-1pc{
&& {\clan{1.,}}\\
{\clan{+-,}}\ar[urr]^1 &&
{\clan{2,.}}\ar[u]^2 && 
{\clan{-+,}}\ar[ull]_1\\
{\clan{+,-}}\ar[u]^2&&
{\clan{,1.}}\ar[u]^1&& 
{\clan{-,+}}\ar[u]_2\\
&{\clan{,+-}}\ar[ur]_2\ar[ul]^1 &&
{\clan{,-+}}\ar[ul]^2\ar[ur]_1
}$$
\end{example}

\begin{definition}\label{length}\cite[Definition 3.7]{S-orbit}
For a $(p,m,q)$-preclan $\gamma$, let $\{a_1<\cdots<a_m\}$ be the set of indices of uncolored nodes of $\gamma$. Define the \emph{length} of $\gamma$ to be 
\begin{equation}\label{eq:lenghts}
\ell(\gamma):= \ell(\bar{\gamma})+\sum_{i=1}^m(a_i-i),
\end{equation}
where $\bar{\gamma}$ is the $(p,0,q)$-preclan obtained from $\gamma$ by ignoring all the uncolored $\clan{,}$'s, and we define $$\ell(\bar{\gamma})=\sum_{(i,j)\text{-matching}} (j-i) 
- \texttt{\#}\{\text{crossings in $\bar{\gamma}$} \}.$$
\end{definition}

\begin{example}
For the $(4,2,5)$-preclan \eqref{eq:preclaneg}.
The $(4,0,5)$-preclan $\bar{\gamma}$ obtained from $\gamma$ is
\begin{equation}\label{eq:claneg}
\bar{\gamma} = \clan{5-63-..+.}.
\end{equation}
We have 
$\ell(\bar{\gamma})=(6-1)+(9-3)+(7-4)-2=12.$
Since the indices of uncolored nodes are $5,9$; we thus have 
$$\ell(\gamma)=12+(5-1)+(9-2)=23.$$
\end{example}

\subsection{The triple support}
A preclan $\gamma$ is called a \emph{Richardson} preclan, if 
\begin{itemize}
    \item the matchings of $\gamma$ are pairwise non-crossing; 
    \item uncolored nodes $\clan{,}$'s of $\gamma$ are not covered by any matchings. 
\end{itemize}

Then we have the following triple Schubert expansion.
\begin{theorem}[\cite{S-orbit}]\label{th:triplecoeffi}
For a Richardson $(p,m,q)$-preclan $\gamma$, we have 
$$
\mathfrak{S}_{u_\gamma}(x;y)\cdot
\mathfrak{S}_{v_\gamma}(x;t)
=\sum_{w\in S_n} c_{u_\gamma,v_\gamma}^w(t;y)\cdot \mathfrak{S}_w(x;t),
$$
where
$$c_{u_\gamma,v_\gamma}^w(t;y)= 
\begin{cases}
\widetilde{F}_{f_{w*\gamma}}(-y_q,\ldots,-y_1;-t_p,\ldots,-t_1,-t_{p+1},\ldots,-t_n), \\\qquad \text{if } \ell(w*\gamma)=\ell(w)+\ell(\gamma) \,\text{ and } v_{w*\gamma}=e, \\[1ex]
0, \quad \text{otherwise}. 
\end{cases}$$
\end{theorem}

Let $v\in S_\infty$ be any $p$-inverse Grassmannian permutation,   and $u\in S_\infty$ be  any $q$-inverse Grassmannian permutation. By   \cite[Theorem 3.16]{S-orbit}, there exists a Richardson preclan $\gamma=\gamma_{u,v}$ such that $u=u_{\gamma},  v=v_{\gamma}$.
 By Theorem~\ref{th:triplecoeffi}, we know that the triple structure constant is nonzero if and only if 
\[
    \ell(w*\gamma)=\ell(w)+\ell(\gamma) \quad\text{and}\quad v_{w*\gamma}=e.
\]
That is, we have
\begin{equation}\label{eq:supp3_inv1}
    w\in \supp^{3}(u,v) \quad\iff\quad \begin{cases}
        \ell(w*\gamma)=\ell(w)+\ell(\gamma),\\
        v_{w*\gamma}=e.
    \end{cases}
\end{equation}

Define an operation on preclans $\gamma \mapsto\cev{\gamma}$ by interchanging all
$\clan{+}$ and $\clan{-}$ signs in the preclan $\gamma$.
Note that if $\gamma$ is a $(p,m,q)$-preclan,  then $\cev{\gamma}$ is a $(q,m,p)$-preclan, and $$u_{\cev{\gamma}}=v_{\gamma},\quad v_{\cev{\gamma}}=u_{\gamma}.$$ 
Moreover,
the operation is compatible with the weak order, that is,
$w*\cev{\gamma}=\cev{w*\gamma}$. Applying Theorem~\ref{th:triplecoeffi} to $\cev{\gamma}$, we obtain that $w\in \supp^{3}(v_{\gamma},u_{\gamma})$ if and only if $\ell(w*\cev{\gamma})=\ell(w)+\ell(\cev{\gamma})$ and $v_{w*\cev{\gamma}}=e$. 
So we have
\begin{equation}\label{eq:supp3_inv2}
    w\in \supp^{3}(v,u) \quad\iff\quad \begin{cases}
        \ell(w*\gamma)=\ell(w)+\ell(\gamma),\\
        u_{w*\gamma}=e.
    \end{cases}
\end{equation}
Combining $\eqref{eq:supp3_inv1}$ and $\eqref{eq:supp3_inv2}$, we have 
\begin{equation}\label{eq:supp3_inv_cap}
    w\in \supp^{3}(u,v)\cap \supp^{3}(v,u)    
    \iff 
    \begin{cases}
    \ell(w*\gamma)=\ell(w)+\ell(\gamma),\\
    v_{w*\gamma}=e \text{ and } u_{w*\gamma}=e.  
    \end{cases}
\end{equation}

\subsection{The double support}
Throughout this
subsection, we assume, without loss of generality, that $p\geq q$.
We first recall the definition of a permutational preclan. 
\begin{definition}[\cite{S-orbit}]\label{def:perm}
A $(p,m,q)$-preclan $\gamma$ is called {\it permutational}, if 
\begin{itemize}
    \item the first $q$ nodes are left-ends; 
    \item the next $p-q$ nodes are $\clan{+}$.
\end{itemize}
In particular, $v_\gamma=e$ and there is no $\clan{-}$ in $\gamma$.    
For a permutational $(p,m,q)$-preclan $\gamma$, we define a permutation $\eta_{\gamma}\in S_{q+m}$ by 
\begin{equation*}
\eta_\gamma(i)=f_\gamma(i+p-q)-p,\quad \text{for $1\leq i\leq q+m$.}
\end{equation*}
\end{definition}
The property of being permutational has the following simpler
characterization.
\begin{lemma}[\cite{S-orbit}]\label{lem:permutational}
    Let $\gamma$ is a $(p,m,q)$-preclan and $p\ge q$. Then $\gamma$ is  permutational if and only if  $u_{\gamma}=e$ and $v_{\gamma}=e$.
\end{lemma}

The following theorem gives the double Schubert expansion associated with a
Richardson preclan.

\begin{theorem}[\cite{S-orbit}]\label{thm:maindouble}
For a Richardson $(p,m,q)$-preclan $\gamma$, we have 
$$
\mathfrak{S}_{u_\gamma}(x;t)\cdot\mathfrak{S}_{v_\gamma}(x;t)=\sum_{w\in S_n} c_{u_\gamma,v_\gamma}^w(t)\cdot \mathfrak{S}_w(x;t),$$
where 
$$c_{u_\gamma,v_\gamma}^w(t)
=\begin{cases}
\mathfrak{S}_{\eta_{w*\gamma}}(-t_q,\ldots,-t_1;-t_{p+1},\ldots,-t_n),\\
\qquad \quad \text{if $\ell(w*\gamma)=\ell(w)+\ell(\gamma)$ and $w*\gamma$ is permutational};\\[1ex]
0, \qquad \text{otherwise}. 
\end{cases}
$$
\end{theorem}
Now let $\gamma=\gamma_{u,v}$ be the Richardson preclan associated with
$u$ and $v$. By Theorem~\ref{thm:maindouble}, the coefficient
$c_{u,v}^{w}(t)$ is nonzero if and only if
\[
\ell(w*\gamma)=\ell(w)+\ell(\gamma)\quad \text{and\quad $w*\gamma$ is permutational.}
\]
Consequently,
\begin{equation}
\label{eq:supp2_inv}
w\in \supp^{2}(u,v)
\quad\Longleftrightarrow\quad
\begin{cases}
\ell(w*\gamma)=\ell(w)+\ell(\gamma),\\
w*\gamma \text{ is permutational}.
\end{cases}
\end{equation}

\begin{proof}[Proof of Theorem \ref{thm:inverse grass cases}]
Combining \eqref{eq:supp3_inv_cap}, \eqref{eq:supp2_inv} and Lemma~\ref{lem:permutational}, we get the proof of Theorem \ref{thm:inverse grass cases}.
\end{proof}

\begin{example}
    For our running example $u=312$ and $v=231$, we have 
    \begin{align*}
    \supp^3(u,v)& =\{s_1s_2,s_2s_1s_2,s_3s_2s_1s_2\},\\
    \supp^3(v,u)&=\{s_2s_1,s_1s_2s_1,s_3s_2s_1,s_3s_2s_1s_2\},\\
    \supp^2(u,v)&=\{s_3s_2s_1,s_3s_2s_1s_2\},
    \end{align*}
    as shown below (where the framed preclans $\gamma$ satisfy $v_\gamma= e$)
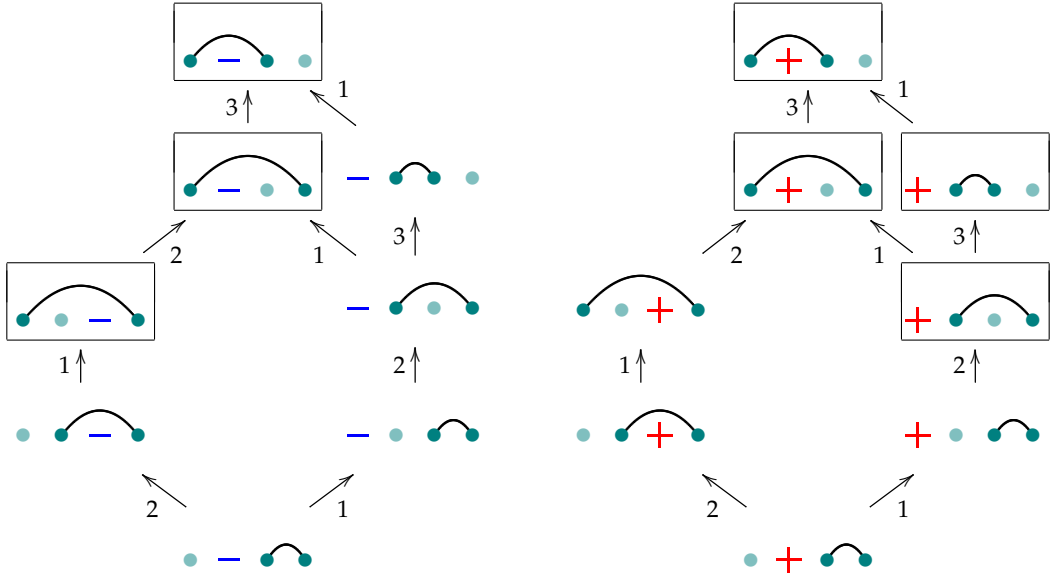
   $$\begin{array}{cc}
\begin{matrix}
\scalebox{1}{\xymatrix@=1pc{&
\fmyclan{2-.,}&\\
&
\fmyclan{3-,.}\ar[u]^{3} &
\myclan{-1.,}\ar[ul]_1 \\
\fmyclan{3,-.} \ar[ur]_2&&
\myclan{-2,.} \ar[ul]^1\ar[u]^3\\
\myclan{,2-.}\ar[u]^1 &&
\myclan{-,1.}\ar[u]^2 \\& 
\myclan{,-1.}\ar[ul]^2\ar[ur]_1}}
\end{matrix}
   \qquad
   &
\begin{matrix}
\scalebox{1}{\xymatrix@=1pc{&
\fmyclan{2+.,}&\\
&
\fmyclan{3+,.}\ar[u]^{3} &
\fmyclan{+1.,}\ar[ul]_1 \\
\myclan{3,+.} \ar[ur]_2&&
\fmyclan{+2,.} \ar[ul]^1\ar[u]^3\\
\myclan{,2+.}\ar[u]^1 &&
\myclan{+,1.}\ar[u]^2 \\& 
\myclan{,+1.}\ar[ul]^2\ar[ur]_1}}
\end{matrix}
   \end{array}
$$
These agree with Example \ref{eg:runningeg}. 
\end{example}

\section{Horn's Inequalities and Triple Saturation Property}\label{section6}

In this section, we will restrict to Grassmannian permutations. 
Fix $k\geq 0$, and let $\mathbb{P}$ be the set of partitions of length at most $k$. For partitions $\lambda=(\lambda_1,\ldots,\lambda_k)$ and $\mu=(\mu_1,\ldots,\mu_k)$ in $\mathbb{P}$, we write $\lambda\leq \mu$ if $\lambda_i\leq \mu_i$ for $i=1,\ldots,k$, i.e., the Young diagram of $\lambda$ is contained in that of $\mu$.
For a partition $\lambda=(\lambda_1,\ldots,\lambda_k)\in \mathbb{P}$, define 
$w_\lambda$ to be the corresponding $k$-Grassmannian permutation defined as
\[
w_{\lambda}(k+1-i)=\lambda_i+k+1-i, \qquad \text{for   $i=1,\ldots,k$},
\]
and the last $n-k$ elements of $w_{\lambda}$ are arranged increasingly. It is well known that $w_{\lambda}\le w_{\mu}$ if and only if $\lambda\subseteq \mu$, and
$\mathfrak{S}_{w_{\lambda}}(x)$ is the Schur polynomial $s_{\lambda}(x)$. For partitions $\lambda,\mu$ and $\nu$, the single, double, and triple LR-coefficients $c_{\lambda,\mu}^{\nu}$,   $c_{\lambda,\mu}^{\nu}(t)$,  $c_{\lambda,\mu}^{\nu}(t;y)$ are respectively defined as:
\begin{align}
    s_{\lambda}(x)\cdot s_{\mu}(x)&=\sum_{\nu}c_{\lambda,\mu}^{\nu}\cdot s_{\nu}(x),\label{eq:singleLR}\\
    s_{\lambda}(x;t)\cdot s_{\mu}(x;t)&=\sum_{\nu}c_{\lambda,\mu}^{\nu}(t)\cdot s_{\nu}(x;t).\label{eq:doubleLR}\\
    s_{\lambda}(x;y)\cdot s_{\mu}(x;t)&=\sum_{\nu}c_{\lambda,\mu}^{\nu}(t;y)\cdot s_{\nu}(x;t).\label{eq:triLR}
\end{align}

For $i=1,2,3,$  denote 
$$\supp^i(\lambda,\mu)=\{\nu\in \mathbb{P}: w_\nu\in \supp^i(w_\lambda,w_\mu)\}.$$
Then Theorem \ref{th:puzzle} can be written as 
\begin{equation}\label{eq:th4.3}
\supp^2(\lambda,\mu)
=
\supp^3(\lambda,\mu)\cap \{\nu\in \mathbb{P}:\nu\geq \lambda\}. 
\end{equation}
Similarly, Theorem \ref{thm:i->3} can be written as 
\begin{equation}\label{eq:th1.1}
\supp^3(\lambda,\mu)
= \bigcup_{\lambda'\leq \lambda}\supp^i(\lambda',\mu),\quad \text{for}\, i=1,2,3.
\end{equation}

 For a partition $\lambda$ and an integer $N>0,$ let $N\lambda$ be the partition obtained from $\lambda$ by multiplying $N$ on each part of $\lambda$. 
 Knutson and Tao \cite{KT} established the saturation theorem of $c_{\lambda,\mu}^{\nu}$ by using honeycomb models.
 Anderson, Richmond and Yong \cite{ARY} obtained the double version of the saturation theorem of  $c_{\lambda,\mu}^{\nu}(t)$ via edge labeled tableaux models.
Moreover, the non-vanishing   of both  single and  double LR-coefficients  are controlled by the Horn's inequalities. 

\begin{theorem}[\cite{KT}]
\label{thm:saturation1}
For  partitions $\lambda,\mu,\nu$ and any integer $N>0$, we have
   \[\text{ $c_{\lambda,\mu}^{\nu}\neq 0$\quad if and only if \quad $c_{N\lambda,N\mu}^{N\nu}\neq 0$.}\]
\end{theorem}

\begin{theorem}[\cite{ARY}]
\label{thm:saturation2}
For  partitions $\lambda,\mu,\nu$ and any   integer $N>0$, we have
\[\text{     $c_{\lambda,\mu}^{\nu}(t)\neq 0$\quad if and only if \quad $c_{N\lambda,N\mu}^{N\nu}(t)\neq 0$. }\]
\end{theorem}

\begin{theorem}[\cite{ARY}]\label{yongthm1.3}
Let $\lambda,\mu,\nu\in \mathbb{P}$ such that $\lambda\leq \nu$ and $\mu\le \nu$. 
The following are equivalent:
\begin{enumerate}
    \item  
    $c_{\lambda,\mu}^{\nu}(t)\neq 0$. 

    \item For every $d<r$, and every triple of  subsets $I,J,K\subseteq [r]$
    of cardinality $d$ such that $c_{\tau(I),\tau(J)}^{\tau(K)}>0,$
    one has
    \begin{equation*}
    \sum_{i\in I}\lambda_i
    +
    \sum_{j\in J}\mu_j
    \geq
    \sum_{k\in K}\nu_k,
    \end{equation*}
where $\tau(K)=(k_d-d,\ldots,k_1-1)$ for a $d$-subset $K=\{k_1<\cdots<k_d\}\subseteq [r].$

    \item There exist $r\times r$ Hermitian matrices $A$, $B$, and $C$ with eigenvalues $\lambda$, $\mu$, and $\nu$, respectively, such that $
    A+B\geq C.$
\end{enumerate}
\end{theorem}

We show that the triple LR-coefficients $c_{\lambda,\mu}^{\nu}(t;y)$ also have saturation property. Moreover,  the Horn's inequalities also control the non-vanishing of triple LR-coefficients.

\begin{theorem}\label{triplesaturation}
For  partitions $\lambda,\mu,\nu$ and any   integer $N>0$, we have 
 \[\text{    $c_{\lambda,\mu}^{\nu}(t;y)\neq 0$\quad if and only if \quad $c_{N\lambda,N\mu}^{N\nu}(t;y)\neq 0$.}\]
\end{theorem}

\begin{theorem}\label{thm:saturation3}
Let $\lambda,\mu,\nu\in \mathbb{P}$ such that $\mu\leq \nu$. 
The following are equivalent:
\begin{enumerate}
    \item 
    $c_{\lambda,\mu}^{\nu}(t;y)\neq 0$. 

    \item For every $d<r$, and every triple of  subsets $I,J,K\subseteq [r]$
    of cardinality $d$ such that $c_{\tau(I),\tau(J)}^{\tau(K)}>0,$
    one has
    \begin{equation}
    \sum_{i\in I}\lambda_i
    +
    \sum_{j\in J}\mu_j
    \geq
    \sum_{k\in K}\nu_k,
    \end{equation}
where $\tau(K)=(k_d-d,\ldots,k_1-1)$ for a $d$-subset $K=\{k_1<\cdots<k_d\}\subseteq [r].$

    \item There exist $r\times r$ Hermitian matrices $A$, $B$, and $C$ with eigenvalues $\lambda$, $\mu$, and $\nu$, respectively, such that $
    A+B\geq C.$
\end{enumerate}
\end{theorem}

\begin{proof}
\noindent
$(1)\Rightarrow(3)$. 
By  \eqref{eq:th1.1}, there exists a partition $\lambda'\leq \lambda$ such that $\nu\in \supp^1(\lambda',\mu)$. 
By Horn's Theorem, 
there exist $n\times n$ Hermitian matrices $A'$, $B$, and $C$ with eigenvalues $\lambda'$, $\mu$, and $\nu$, respectively, such that $A'+B=C$. 
By diagonalizing $A'$, we can find an Hermitian matrix $A$ with eigenvalues $\lambda$ such that $A'\leq A$. Therefore, $A+B\geq A'+B=C$. 

\medskip

\noindent
$(3)\Rightarrow(1)$. By \cite[Theorem~5, $(2)\Leftrightarrow(6)$]{Fulton00}, there exist partitions $\lambda'\leq \lambda$ and $\mu'\leq \mu$ such that $\nu\in \supp^1(\lambda',\mu')$. 
By \eqref{eq:1<2<3} and \eqref{eq:th1.1}, 
we have 
$$\nu\in 
\supp^1(\lambda',\mu')
= \supp^1(\mu',\lambda')
\subseteq \supp^3(\mu',\lambda')
\subseteq \supp^3(\mu,\lambda').$$
Since we assumed $\mu\leq \nu$, 
by  \eqref{eq:1<2<3} and \eqref{eq:th4.3}, we indeed have 
\begin{align*}
    \nu&\in 
\supp^{3}(\mu,\lambda')\cap \{\nu\in \mathbb{P}:\nu\ge \mu\}\\
&=\supp^2(\mu,\lambda')
=\supp^2(\lambda',\mu)\subseteq \supp^3(\lambda',\mu)
\subseteq \supp^3(\lambda,\mu).
\end{align*}

\noindent
$(2)\Leftrightarrow(3)$.
This equivalence  follows directly from  \cite{Fri00} and \cite[Theorem 1]{Fulton00}.
\end{proof}

By Lemma \ref{fangda1}, $c_{\lambda,\mu}^{\nu}(t;y)=0$ unless $\mu\leq \nu$; see also \cite[Theorem~3.1]{MS}.
So Theorem \ref{thm:saturation3} completely characterizes the set $\supp^3(\lambda,\mu)$. 
Moreover, Theorem \ref{yongthm1.3} (\cite[Theorem~1.3]{ARY}) follows from \eqref{eq:th4.3}.

\begin{proof}[Proof of Theorem \ref{triplesaturation}]
This is an  immediate corollary of the equivalence (1) $\Leftrightarrow$ (3) above. 
\begin{align*}
\nu\in \supp^3(\lambda,\mu)\iff & 
\nu\geq \mu \text{ and there exists Hermitian matrices $A,B,C$}\\
& \text{with eigenvalues $\lambda,\mu,\nu$ such that $A+B\geq C$}\\
\iff &
N\nu\geq N\mu \text{ and there exists Hermitian matrices $A',B',C'$}\\
& \text{with eigenvalues $N\lambda,N\mu,N\nu$ such that $A'+B'\geq C'$}\\
\iff &N\nu\in \supp^3(N\lambda,N\mu). 
\qedhere
\end{align*}
\end{proof}

\end{document}